\documentclass[11pt]{amsart}
\usepackage[T1]{fontenc}
\usepackage[utf8]{inputenc}
\usepackage{lmodern,microtype}
\usepackage{amsmath,amssymb,mathtools}
\usepackage{booktabs,array,enumitem,needspace}
\usepackage{xurl,xcolor}
\usepackage[letterpaper,margin=1.02in]{geometry}
\usepackage{hyperref}
\hypersetup{colorlinks=true,linkcolor=blue!45!black,citecolor=blue!45!black,urlcolor=blue!45!black,pdftitle={Sectional category versus module sectional category for retractive maps},pdfauthor={Paul-Eugène Parent}}
\allowdisplaybreaks[2]
\newcommand{\Q}{\mathbb Q}
\newcommand{\secat}{\operatorname{secat}}
\newcommand{\msecat}{\operatorname{msecat}}
\newcommand{\id}{\mathrm{id}}

\newcommand{\Span}{\operatorname{span}_{\Q}}
\newcommand{\Ob}{\mathcal O}
\newcommand{\Cc}{\mathcal C}
\newcommand{\Ahat}{\widehat A}
\newcommand{\Khat}{\widehat K}
\newcommand{\Phat}{\widehat P}
\newcommand{\Rhat}{\widehat R}
\newcommand{\m}{\mathfrak m}
\DeclareRobustCommand{\fundingmark}{\textsuperscript{*}}
\newtheorem{theorem}{Theorem}[section]
\newtheorem{proposition}[theorem]{Proposition}
\newtheorem{lemma}[theorem]{Lemma}
\newtheorem{corollary}[theorem]{Corollary}
\theoremstyle{definition}

\theoremstyle{remark}

\title[Sectional category for retractive maps]{Sectional category versus module sectional category for retractive maps}
\author[Paul-Eugène Parent]{Paul-Eugène Parent\fundingmark}
\thanks{\makebox[0pt][r]{\fundingmark\enspace}This research was supported by the Natural Sciences and Engineering Research Council of Canada (NSERC), grant no.~149521.}
\address{Department of Mathematics and Statistics, University of Ottawa, Ottawa, Canada}
\email{pparent@uottawa.ca}
\date{September 7, 2026}
\AtEndDocument{
\bigskip
\noindent\textbf{Declaration on the use of generative artificial intelligence.}
During the preparation of this manuscript, the author used ChatGPT
(OpenAI) as an assistive tool for language editing, \LaTeX{} preparation,
and the development and review of mathematical arguments. The author
assumes full responsibility for all mathematical statements, proofs,
references, conclusions, and the content of the article.
}
\begin{document}
\begin{abstract}
We construct a map between simply connected rational spaces of finite rational type which admits a homotopy retraction and satisfies $\msecat(p)=2<\secat(p)=3$. The target is the rationalization of a simply connected twelve-dimensional CW complex with eleven cells, and is rationally hyperbolic. This disproves Carrasquel--Vera's conjecture for retractive maps.
\end{abstract}
\maketitle

\section{Introduction}\label{sec:introduction}
The normalized sectional category $\secat(p)$ of a map $p:X\to Y$ is the least integer $n$ for which $Y$ is covered by $n+1$ open sets admitting homotopy local sections. A global homotopy section therefore gives sectional category zero. This invariant originates in the work of Schwarz \cite{Schwarz1966}.

Its rational module approximation $\msecat(p)$, introduced by Fernández Suárez, Ghienne, Kahl and Vandembroucq \cite{FGKV2006}, replaces multiplicative retractions of fibrewise-join models by retractions of differential modules. Hess's theorem gives equality for rational Lusternik--Schnirelmann category \cite{Hess1991}. For general maps, Stanley's spherical-fibration example gives a strict inequality \cite{Stanley2000}; its module interpretation is explained in \cite[Remark~5.8]{FGKV2006}.

Carrasquel--Vera considered maps whose rationalizations admit homotopy retractions. Here a homotopy retraction of $p:X\to Y$ means a map $r:Y\to X$ such that $rp\simeq\id_X$. Such a retractive map has a split Sullivan representative
\[
 (\Lambda V,d)\longrightarrow
 (\Lambda V\otimes\Lambda W,D)
 \xrightarrow{\varphi}(\Lambda V,d),
 \qquad \varphi|_{\Lambda V}=\id,\qquad \varphi(W)=0.
\]
Put $A=(\Lambda V\otimes\Lambda W,D)$ and $K=\ker\varphi$. Carrasquel--Vera characterized the two invariants by homotopy retractions of the projections $\rho_n:A\to A/K^{n+1}$. For $\secat$, a retraction is required in the homotopy category of commutative differential graded algebras; for $\msecat$, it is required in the derived category of differential $A$-modules. He conjectured equality in this retractive setting \cite[Theorem~3.3 and Conjecture~3.4]{CV2017}.

The following theorem disproves this conjecture: a homotopy splitting does not suffice to make the module approximation exact. A quadratic compatibility condition required for a multiplicative retraction can remain undetected by the module retraction.

\begin{theorem}\label{thm:main}
There is a map $p:X\to Y$ admitting a homotopy retraction and satisfying
\[
 \msecat(p)=2<\secat(p)=3.
\]
Both spaces are simply connected and rational, of finite rational type. The source has minimal Sullivan model
\[
 (\Lambda(a,b),d),\qquad |a|=6,\quad |b|=17,\qquad
 da=0,\quad db=a^3.
\]
The target has finite-dimensional rational cohomology, with
\[
 H^q(Y;\Q)=0\quad(q>12),\qquad \dim_{\Q}H^{12}(Y;\Q)=2.
\]
\smallskip
\noindent
The target is the rationalization of a simply connected twelve-dimensional CW complex with eleven cells. The space $Y$ is rationally hyperbolic.
\end{theorem}

The distinction between the two retractions is expressed by a quadratic compatibility condition on a primitive. We construct a semifree resolution of $K^3$ and a module nullhomotopy of its inclusion into the total Sullivan algebra. This proves the upper bound $\msecat(p)\le2$. The module homotopy is permitted to have a component in the base algebra. A coefficient functional is chosen so that all module obstructions can be removed while a square in the base remains nonzero.

A multiplicative retraction is subject to an additional requirement: the image of a square must be the square of a single image. A relative Sullivan algebra with seven generators detects this requirement. One of its differential equations, evaluated by the coefficient functional, forces the base component of a primitive to be $-a/2$. Applying the split projection to the same equation forces that component to vanish. These two conclusions exclude every multiplicative retraction.

The construction is organized around a simultaneous extension lemma for Sullivan algebras and semifree modules. An elementary resolution of $(x,y)^3$ supplies the initial module, and degree considerations identify all terms relevant to the quadratic obstruction. After one protected extension in degree twelve, the construction continues uniformly in all higher degrees.

The target also has an explicit finite CW representative with eleven cells, including the zero-cell. Starting with the four-skeleton of $\mathbb{CP}^2\times\mathbb{CP}^2$, one adds a six-dimensional sphere and the two eight-dimensional product cells. Two twelve-dimensional cells are then attached rationally by the half Whitehead squares of $c$ and $c+\lambda$, where $c$ represents the six-dimensional sphere and $\lambda$ is a specified Whitehead product in the four-skeleton. Appendix~\ref{app:cellular} derives these attachments from a finite-dimensional quotient of the constructed Sullivan model and displays the inclusion and projection that realize the splitting in Quillen models.

Throughout, $Y$ denotes the rationalization of this finite CW complex. Its hyperbolicity forces infinitely many generators in every minimal Sullivan model, although its minimal Quillen model has ten generators. Finite rational type means finite dimensionality of each rational cohomology group; the construction uses a degreewise finite Sullivan algebra.

Section~\ref{sec:models} recalls the algebraic retraction criteria. Section~\ref{sec:attachment} proves the extension lemma. The initial system and its low-degree extensions are constructed in Sections~\ref{sec:initial} and~\ref{sec:detector}. Section~\ref{sec:completion} completes the model and establishes the module value. Section~\ref{sec:multiplication} proves the multiplicative obstruction and realizes the required map. Appendix~\ref{app:exact-value} determines the exact sectional category using a general lemma on module and algebra retractions. Appendix~\ref{app:cellular} gives the finite quotient, the split Quillen model, and the eleven-cell representative.

\section{Sullivan models and module retractions}\label{sec:models}
We work over $\Q$, with cohomological differentials of degree one. Our conventions for Sullivan algebras and differential modules are those of Félix, Halperin and Thomas \cite{FHT2001}. For a graded vector space $V$, the notation $\Lambda V$ denotes the free graded-commutative algebra on $V$. Thus even generators are polynomial and odd generators are exterior. We write $\Lambda^jV$ for words of length $j$; the superscript in $(\Lambda V)^q$ instead denotes cohomological degree.

A relative Sullivan algebra is an inclusion
\[
 (\Lambda V,d)\hookrightarrow(\Lambda V\otimes\Lambda W,D)
\]
with an ordered basis of $W$ whose differentials use only earlier relative generators and elements of $\Lambda V$. It is split if the projection $\varphi$ which fixes $\Lambda V$ and annihilates $W$ is a morphism of differential algebras. We abbreviate
\[
 B=(\Lambda V,d),\qquad A=(\Lambda V\otimes\Lambda W,D).
\]
The kernel and its powers are
\begin{equation}\label{eq:kernel}
 K=\Lambda V\otimes\Lambda^+W,\qquad
 K^j=\Lambda V\otimes\Lambda^{\ge j}W.
\end{equation}
In particular, base factors do not contribute to the relative word length.

Throughout the construction of the split model $\Ahat$, the Sullivan algebras are connected, simply connected and of finite type. For algebras and modules, finite type means finite dimensionality of each homogeneous component. The base is minimal, and relative generators are adjoined in nondecreasing cohomological degree. Such a split extension is minimal as an absolute Sullivan algebra: a linear term in $Dw$, for $|w|=n$, would be either a base term, excluded by $\varphi(Dw)=0$, or an earlier relative generator of degree $n+1$, excluded by the degree order.

A semifree differential $A$-module is a module
\[
 P=(A\otimes E,\delta)
\]
with an ordered basis of $E$ such that the differential of each basis element uses only earlier basis elements. We call these basis elements cells. The differential satisfies
\[
 \delta(cp)=D(c)p+(-1)^{|c|}c\,\delta p.
\]
A homogeneous $A$-linear map of degree $-1$ satisfies
\[
 h(cp)=(-1)^{|c|}c\,h(p).
\]
The derived category of differential $A$-modules is denoted by $\mathcal D(A)$.

\begin{lemma}\label{lem:derived}
Let $J$ be a differential ideal in $A$, and let $\varepsilon:P\xrightarrow{\simeq}J$ be a semifree resolution. The projection $A\to A/J$ has a left inverse in $\mathcal D(A)$ if and only if there is an $A$-linear map $h:P\to A$ of degree $-1$ with
\begin{equation}\label{eq:homotopy}
 Dh+h\delta=\iota\varepsilon,
\end{equation}
where $\iota:J\hookrightarrow A$ is the inclusion.
\end{lemma}
\begin{proof}
The short exact sequence $0\to J\to A\to A/J\to0$ gives a distinguished triangle in $\mathcal D(A)$. Applying morphisms into $A$, the identity of $A$ lifts across the projection precisely when its composite with $J\to A$ is zero. A morphism from a semifree module represents zero in the derived category precisely when it is cochain nullhomotopic. This is \eqref{eq:homotopy}.
\end{proof}

We use the following form of Carrasquel--Vera's characterization.

\begin{theorem}[{\cite[Theorem~3.3]{CV2017}}]\label{thm:CV}
Let $\varphi:A\to B$ be a split Sullivan representative of a map $p$ whose rationalization $p_0$ admits a homotopy retraction, and put $K=\ker\varphi$. Then:
\begin{enumerate}[label=\textup{(\roman*)}]
\item $\secat(p_0)$ is the least $n$ for which $\rho_n:A\to A/K^{n+1}$ admits a homotopy retraction as a differential algebra;
\item $\msecat(p_0)$ is the least $n$ for which $\rho_n$ admits a homotopy retraction as a differential $A$-module.
\end{enumerate}
\end{theorem}

For a map with Sullivan source, a homotopy retraction can be represented by a strict retraction of a relative Sullivan model, as in \cite[p.~806]{CV2017}. To recall the argument, let $i:A\hookrightarrow C$ be a relative Sullivan inclusion, let $g:C\to A$ be a morphism, and let
\[
 H:A\longrightarrow A\otimes\Lambda(t,dt),\qquad |t|=0,\quad |dt|=1,
\]
be a Sullivan homotopy from $gi$ to $\id_A$. The evaluation $\operatorname{ev}_0:A\otimes\Lambda(t,dt)\to A$ is a surjective quasi-isomorphism. The lifting property of $i$ therefore gives
\[
 \widetilde H:C\longrightarrow A\otimes\Lambda(t,dt),\qquad
 \widetilde H i=H,\qquad \operatorname{ev}_0\widetilde H=g.
\]
It follows that $\operatorname{ev}_1\widetilde H$ fixes $A$ strictly; see also \cite[Sections~12 and~14]{FHT2001}. We use this only for a hypothetical retraction of $\rho_2$.

\section{Simultaneous extensions}\label{sec:attachment}
The module $P$ records generators of the ideal $K^3$, their relations, and higher relations. Its augmentation is a resolution precisely when the augmentation cone is acyclic. A nonzero cone cohomology class records a relation still to be resolved. The map $\Ob$ below measures the obstruction to resolving such a class while extending the chosen module homotopy. The extension lemma removes this obstruction by adjoining both algebra generators and module cells.

Suppose $\varepsilon:(P,\delta)\to(K^3,D)$ is a semifree module map, and $h:P\to A$ satisfies \eqref{eq:homotopy}. We suppress the inclusion $\iota$ in that identity. Define the augmentation cone by
\begin{equation}\label{eq:cone}
 \Cc(P)^q=(K^3)^q\oplus P^{q+1},\qquad
 \partial(j,p)=(Dj+\varepsilon p,-\delta p).
\end{equation}
The augmentation is a quasi-isomorphism if and only if this cone is acyclic. The map
\begin{equation}\label{eq:obstruction}
 \Ob:\Cc(P)\longrightarrow A,\qquad \Ob(j,p)=-j-h(p),
\end{equation}
is a cochain map, since
\[
 D\Ob(j,p)=-Dj-\varepsilon p+h\delta p
          =\Ob(Dj+\varepsilon p,-\delta p).
\]

\begin{lemma}[Simultaneous extension]\label{lem:extension}
Let $q\ge9$. Suppose $B\to A\to B$ is a minimal split Sullivan algebra of finite type, with $A^0=\Q$, $A^1=0$, and relative generators of degrees between $2$ and $q-2$. Suppose $P$ has finitely many cells, all of degrees between $6$ and $q-1$, and
\[
 H^j(\Cc(P))=0\quad(j<q),\qquad Dh+h\delta=\varepsilon.
\]
Let $\mathcal U$ be a vector subspace of $H^q(A)$ containing the image of $H^q(\Ob)$, and suppose every class of $\mathcal U$ has a cocycle representative in $K$.

One can adjoin relative algebra generators of degree $q-1$ and module cells of degree $q$ so that the splitting and all structural identities extend, the new cone is acyclic through degree $q$, and
\[
 H^j(A)\xrightarrow{\cong}H^j(A')\quad(j<q),\qquad
 H^q(A')\cong H^q(A)/\mathcal U.
\]
The enlarged algebra remains minimal and the enlarged module remains semifree.
\end{lemma}
\begin{proof}
Choose cocycles $w_1,\ldots,w_t\in K^q$ representing a basis of $\mathcal U$, and set
\[
 A'=A\otimes\Lambda(z_1,\ldots,z_t),\qquad
 |z_a|=q-1,\quad Dz_a=w_a,\quad\varphi(z_a)=0.
\]
The differentials are decomposable: they have no pure base component, and there is no old relative generator of degree $q$.

In degrees at most $q$, the quotient complex $A'/A$ is concentrated in degree $q-1$, on the classes of the $z_a$. There is no degree-$q$ term because $A^1=0$. Its connecting map sends $z_a$ to $[w_a]$ and is injective. The cohomology exact sequence gives the asserted preservation below $q$ and the quotient in degree $q$.

Next choose cone cycles $\gamma_i=(j_i,p_i)$ representing a basis of $H^q(\Cc(P))$. This is a second, independent choice of basis: the $[w_a]$ prescribe the algebra cohomology to be killed, whereas the $[\gamma_i]$ prescribe the remaining relations in the module resolution. Each $\Ob(\gamma_i)$ now has a primitive $v_i\in A'^{q-1}$. Extend scalars from $A$ to $A'$, and adjoin cells $g_i$ of degree $q$ with
\begin{equation}\label{eq:cell-extension}
 \delta g_i=p_i,\qquad \varepsilon(g_i)=-j_i,\qquad h(g_i)=v_i,
 \qquad Dv_i=-j_i-h(p_i).
\end{equation}
The cone-cycle equations $\delta p_i=0$ and $Dj_i+\varepsilon p_i=0$ give
\[
 \delta^2g_i=0,\qquad
 D\varepsilon(g_i)=\varepsilon\delta(g_i),\qquad
 Dh(g_i)+h\delta(g_i)=\varepsilon(g_i).
\]
The old cells precede the new ones, so the module is semifree.

A new monomial in $(\ker(A'\to B))^3$ involving a $z_a$ has degree at least $(q-1)+2+2=q+3$. A new scalar multiple of an old cell has cone degree at least $(q-1)+6-1=q+4$. Thus scalar extension does not alter the cone in the degrees needed to compute $H^q$.

The quotient of the enlarged cone by the old cone has its first new generators in degree $q-1$, represented by the cells $g_i$. This second connecting map has values $-[\gamma_i]$, a basis of the old cone cohomology in degree $q$. The first connecting map controlled algebra cohomology; this one controls exactness of the resolution. There is no quotient term in degree $q$, again because $A'^1=0$. Its cohomology exact sequence therefore gives acyclicity through degree $q$. Every choice was made in a finite-dimensional vector space.
\end{proof}

If $H^q(B)=0$, the representative hypothesis of the lemma is automatic. For a cocycle $w$, write $\varphi(w)=dv$ and replace $w$ by $w-D\sigma(v)$. Its class is unchanged and its projection is zero. The formulation of Lemma~\ref{lem:extension} also allows the one degree in which the base cohomology must be retained.

\section{The initial Sullivan algebra and its module}\label{sec:initial}
Fix
\begin{equation}\label{eq:base}
 B=(\Lambda(a,b),d),\qquad |a|=6,\quad |b|=17,\qquad
 da=0,\quad db=a^3.
\end{equation}
Its cohomology is the truncated free graded-commutative algebra
\[
 H^*(B)=\Lambda([a])/([a]^3),
\]
concentrated in degrees $0,6,12$. In particular, $B^{11}=0$ and $[a^2]\ne0$.

Let $R=(\Lambda(x,y),0)$, where $|x|=|y|=2$, and let $\m=(x,y)\subset R$. We write $R_j=\Lambda^j(x,y)$ for the subspace of polynomial length $j$, of cohomological degree $2j$. Consider
\[
 A_0=\bigl(\Lambda(a,b,x,y,t_0,t_1,t_2,t_3,u_0,u_1,u_2),D\bigr)
\]
with $D|_B=d$, $Dx=Dy=0$, $|t_i|=5$, $|u_i|=6$, and
\begin{equation}\label{eq:seed-algebra}
 Dt_i=x^{3-i}y^i\quad(0\le i\le3),\qquad
 Du_i=yt_i-xt_{i+1}\quad(0\le i\le2).
\end{equation}
The projection $\varphi_0:A_0\to B$ annihilates all relative generators. The consecutive syzygies in \eqref{eq:seed-algebra} give $D^2=0$.

Let $P_0$ be the semifree $A_0$-module on cells $e_0,\ldots,e_3$ of degree $6$ and $f_0,f_1,f_2$ of degree $7$. Define
\begin{equation}\label{eq:seed-module}
\begin{array}{c|ccc}
 &\delta&\varepsilon&h\\ \hline
 e_i&0&x^{3-i}y^i&t_i\\
 f_i&ye_i-xe_{i+1}&0&-u_i+\delta_{i0}a.
\end{array}
\end{equation}
The maps extend by the differential module rule and graded $A_0$-linearity. On each $f_i$,
\[
 Dh(f_i)+h\delta(f_i)=-(yt_i-xt_{i+1})+(yt_i-xt_{i+1})=0.
\]
Together with the equations on $e_i$, this proves $Dh+h\delta=\varepsilon$. The same formulas verify $\delta^2=0$ and $D\varepsilon=\varepsilon\delta$.

The term $a$ in $h(f_0)=-u_0+a$ is the only initial component of the homotopy in the base algebra. It is allowed because $h$ takes values in all of $A_0$, rather than only in $K_0=\ker\varphi_0$. This component will cancel a module obstruction, while the square of a primitive will impose an incompatible condition on an algebra retraction.

\begin{lemma}\label{lem:initial-cone}
The cone of $\varepsilon:P_0\to K_0^3$ is acyclic through degree eight. Its degree-nine cohomology has the basis represented by
\[
 \gamma_{m,i}=(mt_i,-me_i),\qquad
 m\in\{x^2,xy,y^2\},\quad 0\le i\le3.
\]
Moreover, $\Ob(\gamma_{m,i})=0$.
\end{lemma}
\begin{proof}
The monomials $x^3,x^2y,xy^2,y^3$ generate $\m^3$, and their syzygies are freely generated by $ye_i-xe_{i+1}$, $0\le i\le2$. To see generation, reduce a relation modulo $y$: its first coefficient is divisible by $y$. Subtract the corresponding first consecutive syzygy, divide the remaining polynomial relation by $y$, and continue. Independence follows successively from the coefficients of $e_0,e_1,e_2$.

Put $J_0=A_0\m^3$. Since the underlying graded algebra of $A_0$ is free over $R$, the preceding resolution extends to an exact sequence
\[
 0\longrightarrow S\xrightarrow{M}E\xrightarrow{\mu}J_0\longrightarrow0
\]
of differential $A_0$-modules. Here $E$ is free on $e_i$ of degree $6$, $S$ is free on $r_i$ of degree $8$, both carry the coefficient differential, and
\[
 M(r_i)=ye_i-xe_{i+1},\qquad \mu(e_i)=x^{3-i}y^i.
\]
The maps commute with that differential because $x,y$ are closed. The module $P_0$ is the mapping cone of $M$. Our desuspension convention is $f_i=s^{-1}r_i$ and $\delta f_i=M(r_i)$, extended by the differential module rule.

Here the contraction can be written explicitly. Put $\xi_i=ye_i-xe_{i+1}$. The kernel of $P_0\to J_0$ is free as a graded $A_0$-module on the $\xi_i$ and $f_i$, with
\[
 \delta\xi_i=0,\qquad \delta f_i=\xi_i.
\]
Define a degree-$-1$ map $k$ on this kernel by
\[
 k(\xi_i)=f_i,\qquad k(f_i)=0,\qquad
 k(c\eta)=(-1)^{|c|}c\,k(\eta).
\]
Then $\delta k+k\delta=\id$: this holds on each generator, and the terms involving $D(c)$ cancel on a scalar multiple. Thus the kernel is contractible and $P_0\to J_0$ is a quasi-isomorphism. This proves the assertion with the nonzero coefficient differential included.

There is consequently an exact sequence of complexes
\[
 0\longrightarrow\Cc(P_0\to J_0)
 \longrightarrow\Cc(P_0\to K_0^3)
 \longrightarrow K_0^3/J_0\longrightarrow0
\]
whose first term is acyclic. A monomial in the last quotient has at least three relative factors, at most two of which are $x,y$, and at least one of which has degree at least five. Its degree is therefore at least $2+2+5=9$. This proves the initial acyclicity.

In degree nine, the quotient has exactly the monomials $mt_i$ displayed in the statement. Their differentials lie in $J_0$, and the degree-eight quotient is zero. They form a basis of its degree-nine cohomology, hence of the cone cohomology. The pairs $\gamma_{m,i}$ are lifts which satisfy the cone-cycle equations. Finally,
\[
 \Ob(mt_i,-me_i)=-mt_i+h(me_i)=0
\]
because $m$ has even degree and $h(e_i)=t_i$.
\end{proof}

Adjoin one cell $c_{m,i}$ of degree nine for each of these classes:
\begin{equation}\label{eq:degree-nine}
 \delta c_{m,i}=-me_i,\qquad
 \varepsilon(c_{m,i})=-mt_i,\qquad h(c_{m,i})=0.
\end{equation}
There is no algebra extension at this step. Lemma~\ref{lem:extension}, or its cone exact sequence directly, makes the cone acyclic through degree nine.

Next apply Lemma~\ref{lem:extension} for $q=10$ and $q=11$, taking $\mathcal U=H^q(A)$ at each step. These choices are possible because $H^{10}(B)=H^{11}(B)=0$. Denote the resulting system by $(A_-,P_-,\varepsilon,h)$. Its relative algebra generators have degrees
\begin{equation}\label{eq:low-degrees}
 2,5,6,9,10,
\end{equation}
and its module cells have degrees $6,7,9,10,11$. Its cone is acyclic through degree eleven. No choices of bases or primitives beyond those in Lemma~\ref{lem:extension} are imposed.

\section{The quadratic obstruction in degree twelve}\label{sec:detector}
We now identify the two linear conditions that allow the module construction to continue while preserving the multiplicative obstruction. A coefficient is always taken in the ordered monomial basis obtained from the displayed generators and the successive relative Sullivan extensions.

\begin{lemma}\label{lem:a-free}
The differential of every module cell of $P_-$ is free of the base generator $a$.
\end{lemma}
\begin{proof}
This is immediate for the initial cells and the degree-nine cells in \eqref{eq:degree-nine}. The differential of a degree-ten cell lies in $P^{11}$. Its possible terms are $t_je_i$, elements of $R_2f_i$, and $x,y$ times degree-nine cells. None contains $a$.

For a degree-eleven cell, its differential $p$ lies in degree twelve of the module already constructed. Its cells have degrees $6,7,9,10$, so the required coefficient degrees are $6,5,3,2$. The complete list of possible terms is therefore
\[
 (R_3+\Span(a,u_0,u_1,u_2))e_i,\qquad
 t_jf_i,\qquad xg,\ yg,
\]
where $g$ ranges over degree-ten cells; degree-nine cells contribute nothing because $A^3=0$. The only terms containing $a$ are $ae_i$. By the attachment formulas there is $j\in(K^3)^{11}$ such that
\[
 \varepsilon(p)=-Dj.
\]
The element $j$ contains no $a$, since $a$ together with three relative factors has degree at least twelve. It contains no new algebra generator of degree nine or ten either, since two further relative factors would give degree at least thirteen. Therefore $Dj$ is free of $a$.

Every other term of $\varepsilon(p)$ is free of $a$ as well. Indeed, aside from $ae_i$, the possible terms of $p$ are $R_3e_i$, $u_je_i$, $t_jf_i$, and $x,y$ times degree-ten cells. The augmentations of the latter cells lie in $(K^3)^{10}$, which has no $a$-term. The four monomials $ax^{3-i}y^i$ are linearly independent, so the coefficients of $ae_i$ must all vanish.
\end{proof}

The coefficient functional is chosen to annihilate
$h(u_0f_0)=-u_0^2+au_0$. Equal coefficients on $u_0^2$ and $au_0$ produce this cancellation. On a square $(u_0+\alpha a+P_3(x,y))^2$, with $P_3\in R_3$, the same functional has value $1+2\alpha$. Section~\ref{sec:multiplication} will show why such a square is forced by a multiplicative retraction.

For a monomial $M$, write $[M]^*$ for its coefficient functional. Define
\begin{equation}\label{eq:ell}
 \ell=[u_0^2]^*+[au_0]^*:A_-^{12}\longrightarrow\Q.
\end{equation}
This is a linear functional on cochains, not an algebra morphism. Neither monomial is assumed to be a cocycle; the next lemma is what permits $\ell$ to descend to cohomology.

\begin{lemma}\label{lem:ell-boundaries}
The functional $\ell$ satisfies
\[
 \ell(DA_-^{11})=0,\qquad
 \ell(xA_-^{10}+yA_-^{10})=0.
\]
\end{lemma}
\begin{proof}
The possible monomials of degree eleven are
\[
 t_iR_3,\qquad t_iu_j,\qquad at_i,\qquad xz,\ yz,
\]
where $z$ ranges over the relative generators of degree nine. This list follows from \eqref{eq:low-degrees} and $|b|=17$. The differential of each monomial belongs to the ideal $(x,y)$. Neither $u_0^2$ nor $au_0$ belongs to that ideal. The same observation applies to $xA_-^{10}+yA_-^{10}$.
\end{proof}

In particular, $\ell$ induces a functional on $H^{12}(A_-)$. We use the same notation for this induced functional.

\begin{proposition}\label{prop:double-annihilation}
For every degree-twelve cone cycle $\gamma\in\Cc(P_-)$,
\begin{equation}\label{eq:two-annihilations}
 \ell(\Ob\gamma)=0,\qquad \varphi(\Ob\gamma)=0.
\end{equation}
The second equality is an equality of cochains in $B^{12}$.
\end{proposition}
\begin{proof}
Write $\gamma=(j,p)$, with $p\in P_-^{13}$. The complete list of possible coefficients of $p$ is
\[
\begin{array}{c|l}
\text{cell degree}&\text{coefficient}\\ \hline
6&xt_j,\ yt_j\\
7&R_3,\ u_j,\ a\\
9&R_2\\
10&0\\
11&x,\ y.
\end{array}
\]
There is no degree-three coefficient for a degree-ten cell. In particular, the only terms containing $a$ are $af_i$.

The differentials of all remaining terms are free of $a$, by Lemma~\ref{lem:a-free} and the initial algebra differentials. Since $\delta p=0$, the coefficients $\lambda_i$ of the $af_i$ satisfy
\[
 \sum_{i=0}^2\lambda_i(ye_i-xe_{i+1})=0.
\]
Reading the coefficients of $e_0,e_1,e_2$ successively gives $\lambda_0=\lambda_1=\lambda_2=0$.

The ideal component $j\in K^3$ has no $u_0^2$, $au_0$ or $a^2$ coefficient: these monomials have relative lengths two, one and zero. In $h(p)$, every remaining term has a factor $x$ or $y$, except possibly the terms $h(u_jf_i)$. Of those, the only term detected by $\ell$ is
\[
 h(u_0f_0)=-u_0^2+au_0,
\]
which $\ell$ annihilates. This proves the first equality in \eqref{eq:two-annihilations}.

For the second equality, all terms with a relative scalar coefficient project to zero. Each $h(u_jf_i)$ also projects to zero. The only possible source of a pure $a^2$ term would have been $h(af_0)$, and the $af_i$ terms have already been excluded. Finally $\varphi(j)=0$. Hence $\varphi(\Ob\gamma)=0$.
\end{proof}

Set
\begin{equation}\label{eq:protected-subspace}
 \mathcal U_{12}
 =
 \ker\bigl(\ell:H^{12}(A_-)\to\Q\bigr)
 \cap
 \ker\bigl(H^{12}(\varphi):H^{12}(A_-)\to H^{12}(B)\bigr).
\end{equation}
Proposition~\ref{prop:double-annihilation} places every obstruction class in this subspace. Every class of $\mathcal U_{12}$ has a representative in $K$: if $w$ is a representative, then $\varphi(w)$ is exact in $B^{12}$, and $B^{11}=0$ implies $\varphi(w)=0$ strictly. Also $\ell(w)=0$, by Lemma~\ref{lem:ell-boundaries}.

Lemma~\ref{lem:extension} now applies with $q=12$ and $\mathcal U=\mathcal U_{12}$. Let
\[
 (A_{12},P_{12},\varepsilon_{12},h_{12})
\]
denote the resulting system. Its new algebra generators have degree eleven and differentials annihilated by $\ell$; its new module cells have degree twelve. Extending $\ell$ by the same coefficient formula, we obtain
\begin{equation}\label{eq:protected-identities}
 H^{\le12}\bigl(\Cc(P_{12})\bigr)=0,\qquad
 \ell\bigl(DA_{12}^{11}+xA_{12}^{10}+yA_{12}^{10}\bigr)=0.
\end{equation}
The projection remains split, so the class $[a^2]$ remains nonzero in $H^{12}(A_{12})$. This establishes the finite stage needed for completion.

For reference, the four steps after the initial system are summarized below. The second column specifies the subspace killed in the current $H^q(A)$; the last column gives the degree of the new module cells.
\[
\begin{array}{c|c|c|c}
 q&\text{algebra classes killed}&\text{new algebra degrees}&\text{new cell degree}\\ \hline
 9&0&\text{none}&9\\
 10&H^{10}(A)&9&10\\
 11&H^{11}(A)&10&11\\
 12&\mathcal U_{12}&11&12
\end{array}
\]
In particular, the degree-nine step requires no algebra generator of degree eight, and degree twelve is treated by a prescribed subspace rather than by killing all cohomology.

\section{Completion and module sectional category}\label{sec:completion}
Completion here means continuing the construction in every higher degree and taking its ordinary increasing union. Every element belongs to a finite stage; no formal series or topological completion is involved.

Starting with $A^{(12)}=A_{12}$ and $P^{(12)}=P_{12}$, apply Lemma~\ref{lem:extension} successively for every $q\ge13$, taking
\[
 \mathcal U=H^q(A^{(q-1)}).
\]
The representative hypothesis holds because $H^q(B)=0$ for all $q\ge13$. At each step the relative algebra generators have degree $q-1$ and the new module cells have degree $q$. Write
\[
 \Ahat=\bigcup_{q\ge12} A^{(q)},\qquad
 \Phat=\bigcup_{q\ge12} P^{(q)},\qquad
 \Khat=\ker(\widehat\varphi:\Ahat\to B).
\]

\begin{proposition}\label{prop:completion}
The union defines a minimal split Sullivan algebra
\[
 (\Lambda(a,b),d)\longrightarrow
 \Ahat=(\Lambda(a,b)\otimes\Lambda W,D)
 \xrightarrow{\widehat\varphi}(\Lambda(a,b),d)
\]
of finite type, together with a semifree quasi-isomorphism
\[
 \varepsilon:\Phat\xrightarrow{\simeq}\Khat^3
\]
and a degree-$-1$ module map $h:\Phat\to\Ahat$ satisfying $Dh+h\delta=\varepsilon$.
Moreover,
\begin{equation}\label{eq:completed-properties}
 H^q(\Ahat)=0\quad(q>12),\qquad [a^2]\ne0,\qquad
 \ell\bigl(D\Ahat^{11}+x\Ahat^{10}+y\Ahat^{10}\bigr)=0.
\end{equation}
\end{proposition}
\begin{proof}
Each stage is minimal and split by Lemma~\ref{lem:extension}. The differential of every generator is fixed at the stage when it is introduced. Their union is therefore a minimal split Sullivan algebra. The same argument preserves the semifree cell order and the homotopy identities.

In any fixed degree, only finitely many stages can introduce a generator or a cell, and every addition is finite. Hence both the algebra and the module have finite type. A fixed cohomological degree of the cone becomes exact at a finite stage and remains exact thereafter. Since kernels, ideal powers and the structural maps agree with these increasing unions, the cone of $\varepsilon:\Phat\to\Khat^3$ is acyclic.

For each $q\ge13$, the entire target cohomology in degree $q$ is killed, and later stages preserve it. This gives the first assertion of \eqref{eq:completed-properties}. At the first completion step, the degree-twelve generators have differentials representing independent classes in $H^{13}(A_{12})$. The injective connecting map in Lemma~\ref{lem:extension} preserves $H^{12}$, so this step creates no new class there. Subsequent stages also preserve $H^{12}$, and the split projection in any case preserves $[a^2]\ne0$.

No completion generator has degree below twelve. The spaces $\Ahat^{10}$ and $\Ahat^{11}$, together with their differentials, are therefore unchanged from the finite stage. The last identity of \eqref{eq:completed-properties} follows from \eqref{eq:protected-identities}.
\end{proof}

Lemma~\ref{lem:derived} now gives a module homotopy retraction of
\[
 \rho_2:\Ahat\longrightarrow\Ahat/\Khat^3.
\]
There is no such retraction at level one: the cocycle $x^2\in\Khat^2$ is nonexact because $\Ahat^3=0$, and its nonzero cohomology class is killed by $\rho_1$. A retraction would make $H(\rho_1)$ injective. The same class rules out level zero. Thus the module minimum in Theorem~\ref{thm:CV} is exactly two.

The same construction also gives an explicit description of the target cohomology, without any further attachments.

\begin{corollary}\label{cor:cohomology}
The Poincar\'e series of $\Ahat$ is
\[
 \sum_{q\ge0}\dim_{\Q}H^q(\Ahat)\,z^q
 =1+2z^2+3z^4+z^6+2z^8+2z^{12}.
\]
In particular, $H^{12}(\Ahat)$ has dimension two.
\end{corollary}
\begin{proof}
In degrees below ten, a monomial of $A_0$ contains at most one of the generators $t_i,u_i$. The only monomials involving the base generator $a$ are $a,ax,ay$. The differentials of the $t_i$ generate the polynomial ideal $(x,y)^3$. The syzygy resolution in Lemma~\ref{lem:initial-cone} shows that every cocycle among their polynomial multiples is a differential of a polynomial combination of the $u_i$. The injectivity of its syzygy map excludes further cocycles involving the $u_i$. Thus the nonzero cohomology groups below degree ten have bases
\[
\begin{array}{c|ccccc}
 q&0&2&4&6&8\\ \hline
 H^q(A_0)&[1]&[x],[y]&[x^2],[xy],[y^2]&[a]&[ax],[ay].
\end{array}
\]
All algebra extensions preserve these degrees. The steps $q=10,11$ kill the entire cohomology groups in those degrees, and subsequent steps preserve their vanishing.

It remains to identify the degree-twelve quotient. Since
$H^{10}(A_-)=H^{11}(A_-)=0$, choose $g\in A_-^9$ and $v_0,v_1\in A_-^{10}$ with
\[
 Dg=t_0t_1+x^2u_0,\qquad
 Dv_0=-t_1u_0+yg,\qquad Dv_1=-t_0u_0+xg.
\]
The first right-hand side is closed by \eqref{eq:seed-algebra}. Substitution of the equation for $Dg$ shows that the other two are closed, which justifies the successive choices. Put
\[
 \Omega=-\tfrac12u_0^2-yv_1+xv_0.
\]
Directly,
\[
 D\Omega=-u_0(yt_0-xt_1)
          -y(-t_0u_0+xg)+x(-t_1u_0+yg)=0.
\]
Moreover, $\ell(\Omega)=-1/2$ and $\varphi(\Omega)=0$, whereas
$\ell(a^2)=0$ and $\varphi(a^2)=a^2$. Hence
\[
 (\ell,H^{12}(\varphi)):H^{12}(A_-)
       \longrightarrow\Q\oplus H^{12}(B)
\]
is surjective, with kernel $\mathcal U_{12}$. The quotient retained in $H^{12}(A_{12})$ therefore has basis $[a^2],[\Omega]$. Proposition~\ref{prop:completion} preserves it in $\Ahat$ and gives vanishing above degree twelve.
\end{proof}

\section{The multiplicative obstruction and realization}\label{sec:multiplication}
We use a partial relative Sullivan model as a test for a multiplicative retraction. The test need not itself resolve the quotient: after it is extended to a full relative model, any retraction would restrict to the test. The capital letters below denote algebra generators, distinct from the module cells $e_i,f_i$.

Consider the following relative Sullivan algebra over $R=(\Lambda(x,y),0)$:
\[
 T=\bigl(R\otimes\Lambda(E_0,E_1,F,G,H_0,H_1,Q),d_T\bigr).
\]
Its seven generators and differentials are
\begin{equation}\label{eq:test}
\begin{array}{c|c|l}
\text{generator}&\text{degree}&d_T\\ \hline
E_0&5&x^3\\
E_1&5&x^2y\\
F&6&yE_0-xE_1\\
G&9&E_0E_1+x^2F\\
H_0&10&-E_1F+yG\\
H_1&10&-E_0F+xG\\
Q&11&-F^2/2-yH_1+xH_0.
\end{array}
\end{equation}
The symbol $Q$ here names a generator; the ground field is $\Q$.
The first three equations satisfy $d_T^2=0$ immediately. The remaining checks are
\begin{align*}
d_T^2G
 &=x^3E_1-x^2yE_0+x^2(yE_0-xE_1)=0,\\
d_T^2H_0
 &=-x^2yF-yE_0E_1+y(E_0E_1+x^2F)=0,\\
d_T^2H_1
 &=-x^3F-xE_0E_1+x(E_0E_1+x^2F)=0,\\
d_T^2Q
 &=-yE_0F+xE_1F+yE_0F-xyG-xE_1F+xyG=0.
\end{align*}
The signs use that $E_0,E_1$ are odd and $F$ is even. Sending the seven generators to zero defines an augmentation $T\to R/\m^3$.

\begin{proposition}\label{prop:no-retraction}
The projection $\rho_2:\Ahat\to\Ahat/\Khat^3$ has no homotopy retraction as a differential algebra.
\end{proposition}
\begin{proof}
Extend scalars to obtain
\[
 \Ahat\otimes_R T\longrightarrow\Ahat/\Khat^3,
\]
fixing $\Ahat$ modulo $\Khat^3$ and sending the seven generators to zero. Complete this map to a relative Sullivan model of $\rho_2$,
\[
 \Ahat\hookrightarrow
 \Ahat\otimes_R T\otimes\Lambda U
 \xrightarrow{\simeq}\Ahat/\Khat^3.
\]
If $\rho_2$ had a homotopy retraction, that model would admit a retraction strictly under $\Ahat$. Restriction would give a morphism
\[
 f:\Ahat\otimes_R T\longrightarrow\Ahat,\qquad f|_{\Ahat}=\id.
\]
We show that no such $f$ exists.

After the initial stage, the construction adds no algebra generators of degrees at most six. Consequently
\[
 \Ahat^5=\Span(t_0,t_1,t_2,t_3),\qquad
 \Ahat^6=\Span(a,u_0,u_1,u_2)\oplus R_3.
\]
The four monomials $Dt_i$ are linearly independent, as are the three vectors $Du_i=yt_i-xt_{i+1}$. The equations for $E_0,E_1,F$ therefore force, as equalities of cochains,
\begin{equation}\label{eq:images}
 f(E_0)=t_0,\qquad f(E_1)=t_1,\qquad
 f(F)=u_0+\alpha a+P_3(x,y),
\end{equation}
where $\alpha\in\Q$ and $P_3\in R_3$. These describe all possible images; no homotopy normalization is required.

Apply $f$ to the last differential in \eqref{eq:test}:
\begin{equation}\label{eq:quadratic}
 Df(Q)=-\tfrac12 f(F)^2-yf(H_1)+xf(H_0).
\end{equation}
By \eqref{eq:completed-properties}, $\ell$ annihilates the left side and the last two terms. Every term of $f(F)^2$ involving $P_3$ contains $x$ or $y$. The coefficient of $u_0^2$ is one and that of $au_0$ is $2\alpha$. Hence
\begin{equation}\label{eq:alpha1}
 0=-\tfrac12(1+2\alpha),\qquad \alpha=-\tfrac12.
\end{equation}

On the other hand, apply $\widehat\varphi$ to the same equation \eqref{eq:quadratic}. It annihilates $x,y,u_0,P_3$, and $\widehat\varphi f(F)=\alpha a$. Since $|Q|=11$ and $B^{11}=0$, the result is the cochain equality
\begin{equation}\label{eq:alpha2}
 0=d\bigl(\widehat\varphi f(Q)\bigr)=-\tfrac12\alpha^2a^2.
\end{equation}
The monomial $a^2$ is nonzero in $\Lambda(a,b)$, so $\alpha=0$. This contradicts \eqref{eq:alpha1} and proves the proposition.
\end{proof}

\begin{proof}[Proof of Theorem~\ref{thm:main}]
Realize the minimal Sullivan algebras by setting
\[
 X=\langle(\Lambda(a,b),d)\rangle,\qquad
 Y=\langle\Ahat\rangle,\qquad p=\langle\widehat\varphi\rangle.
\]
The realization theorem for simply connected minimal Sullivan algebras of finite type identifies these with rational cochain models of $X$ and $Y$ \cite[Theorem~8.1]{Sullivan1977}; see also \cite{FHT2001}. By contravariance, the inclusion $\sigma:B\hookrightarrow\Ahat$ realizes to $r:Y\to X$, and $\widehat\varphi\sigma=\id_B$ gives $rp\simeq\id_X$.

Proposition~\ref{prop:completion} gives the stated cohomology bound and the nonzero class in degree twelve. Section~\ref{sec:completion} and Theorem~\ref{thm:CV} give $\msecat(p)=2$. Proposition~\ref{prop:no-retraction} and the same characterization give $\secat(p)>2$. This establishes the strict separation. Proposition~\ref{prop:exact-secat} in Appendix~\ref{app:exact-value} supplies the upper bound $\secat(p)\le3$. Theorem~\ref{thm:eleven-cells} and Corollary~\ref{cor:target-hyperbolic} establish the finite CW representative and its rational hyperbolicity.
\end{proof}

\appendix
\section{The exact sectional category}\label{app:exact-value}
The separation established in Sections~\ref{sec:completion} and~\ref{sec:multiplication} already disproves the conjecture. We record here the complementary upper bound $\secat(p)\le3$. It follows from a general degree argument converting a module retraction into an algebra retraction.

\begin{lemma}\label{lem:module-algebra-range}
Let $A$ be a nonnegatively graded commutative differential algebra over $\Q$, with $A^0=\Q$ and $H^{>N}(A)=0$, where $N\ge1$. Let
\[
 i:A\hookrightarrow M=(A\otimes\Lambda U,D)
\]
be a relative Sullivan extension with $U^{<r}=0$, where $r\ge1$ and $2r>N$. If $i$ has a left inverse in $\mathcal D(A)$, then $i$ admits a strict retraction of differential algebras.
\end{lemma}
\begin{proof}
The underlying $A$-module of a relative Sullivan extension is semifree. The derived left inverse is therefore represented by a morphism of differential $A$-modules
\[
 \Rhat:M\longrightarrow A.
\]
The composite $\Rhat i$ induces the identity in cohomology. Since $A^0=\Q$ and there are no negative degrees, $\Rhat(1)=1$. By $A$-linearity, $\Rhat i=\id_A$ strictly.

Choose a vector space complement $L$ of $Z^N(A)$ in $A^N$, and put
\[
 J=L\oplus A^{>N},\qquad C=A/J,\qquad q:A\twoheadrightarrow C.
\]
This is a differential ideal: multiplying $L$ by a positive-degree element gives degree greater than $N$, and $D(L)\subset A^{N+1}$. It is acyclic. In degree $N$, this follows from $L\cap Z^N(A)=0$; in degree $N+1$, one has
\[
 D(L)=D(A^N)=Z^{N+1}(A)
\]
because $H^{N+1}(A)=0$. In higher degrees its cohomology agrees with that of $A$ and vanishes. Thus $q$ is a surjective quasi-isomorphism and $C^{>N}=0$.

The cochain map $q\Rhat:M\to C$ is multiplicative. Indeed, in the graded decomposition
\[
 M=A\oplus M^+,\qquad M^+=A\otimes\Lambda^+U,
\]
the second summand is concentrated in degrees at least $r$. For homogeneous $v,w\in M^+$, both $q\Rhat(vw)$ and $q\Rhat(v)q\Rhat(w)$ have degree at least $2r>N$, so both vanish. Products with a factor in $A$ are respected by $A$-linearity. Bilinearity and $\Rhat(1)=1$ give the assertion.

The commutative square of differential algebras
\[
\begin{array}{ccc}
 A&\xrightarrow{\id_A}&A\\
 {\scriptstyle i}\big\downarrow&&\big\downarrow{\scriptstyle q}\\
 M&\xrightarrow{q\Rhat}&C
\end{array}
\]
has a lift because $i$ is a relative Sullivan inclusion and $q$ is a surjective quasi-isomorphism \cite[Sections~12 and~14]{FHT2001}. The lift $f:M\to A$ satisfies $fi=\id_A$.
\end{proof}

\begin{proposition}\label{prop:exact-secat}
The map realized in Section~\ref{sec:multiplication} satisfies $\secat(p)=3$.
\end{proposition}
\begin{proof}
We retain the completed notation $\Ahat$ and $\Khat=\ker\widehat\varphi$. Every relative generator in $\Khat$ has degree at least two. Hence
\begin{equation}\label{eq:fourth-power-range}
 (\Khat^4)^q=0\qquad(q<8).
\end{equation}
The projection $\rho_3:\Ahat\to\Ahat/\Khat^4$ therefore admits a relative Sullivan model
\begin{equation}\label{eq:rho-three-model}
 \Ahat\overset{i}{\hookrightarrow}
 M=(\Ahat\otimes\Lambda U,D)
 \xrightarrow[\simeq]{\theta}\Ahat/\Khat^4,
 \qquad U^{<7}=0.
\end{equation}
For completeness, this degree bound follows directly from the construction by killing kernel cohomology. The initial map is surjective and its kernel starts in degree eight. Beginning at $q=8$, adjoin generators of degree $q-1$ whose differentials represent a basis of degree-$q$ kernel cohomology, sending the new generators to zero in the quotient.

Through degree $q$, the quotient of the new kernel by the old kernel is concentrated in degree $q-1$. Its degree-$q$ term is zero because the current algebra has no degree-one part and a product of two new generators has degree $2q-2>q$. The connecting map sends the new generators to the chosen independent classes. It preserves kernel acyclicity below $q$ and kills kernel cohomology in degree $q$. The increasing union has acyclic kernel and gives \eqref{eq:rho-three-model}, with no generator below degree seven.

Let $\pi:\Ahat/\Khat^4\to\Ahat/\Khat^3$ be the natural projection. Section~\ref{sec:completion}, using Lemma~\ref{lem:derived}, gives a left inverse $s_2$ of $\rho_2$ in $\mathcal D(\Ahat)$. Since $\pi\rho_3=\rho_2$, we have
\[
 (s_2\pi\theta)i=s_2\rho_2=\id_{\Ahat}
 \quad\text{in }\mathcal D(\Ahat).
\]
Thus $i$ has a module homotopy retraction, represented by a differential $\Ahat$-module map $\Rhat:M\to\Ahat$.

Proposition~\ref{prop:completion} gives $H^{>12}(\Ahat)=0$. Applying Lemma~\ref{lem:module-algebra-range} with
\[
 r=7,\qquad N=12,\qquad 2r=14>12,
\]
produces an algebra retraction of $i$. Theorem~\ref{thm:CV} then gives $\secat(p)\le3$. The same criterion and Proposition~\ref{prop:no-retraction} give $\secat(p)>2$, so equality follows.
\end{proof}

The degree bound explains the disappearance of the remaining obstructions. The module retraction already solves the linear constraints, while products of two new generators in \eqref{eq:rho-three-model} begin in degree fourteen, beyond the cohomology of $\Ahat$. The truncation in Lemma~\ref{lem:module-algebra-range} makes this argument on cochains without identifying cohomological vanishing with the vanishing of the cochain spaces themselves.

\section{An eleven-cell target and its split Quillen model}\label{app:cellular}
We identify the rational target $Y=\langle\Ahat\rangle$ with the rationalization of a finite CW complex $Y_{\mathrm{fin}}$ having eleven cells. The proof first replaces $\Ahat$ by a finite-dimensional quotient, then computes its minimal Quillen model and the two morphisms exhibiting the splitting, and finally identifies the cell attachments. Throughout this appendix, $A,K,\Ahat,\Khat$ retain their algebraic meanings from the main text. The symbols $Y_2,Y_4,Y_8$ denote stages of a CW construction. Differentials of commutative algebras have degree $+1$; differentials of Lie models have degree $-1$. A cell of dimension $n$ gives a Lie generator of degree $n-1$, whereas a cycle of Lie degree $k$ represents a class in $\pi_{k+1}$ after rationalization.

The quotient $C$ below has total cochain dimension twenty-five. Its seven differential pairs leave eleven dimensions of cohomology, including the unit. The reduced dual gives twenty-four Lie cobar generators; eliminating seven contractible pairs leaves ten generators in the minimal Quillen model. The full differential, rather than these dimension counts alone, determines the attaching maps.

\subsection{A finite-dimensional quotient}
Retain $R=(\Lambda(x,y),0)$ and write $R_j$ for its polynomials of polynomial degree $j$, hence cohomological degree $2j$. Put
\[
 \mathcal P_7=\Span\{xt_0,xt_1,xt_2,xt_3,yt_3\}\subset\Ahat^7.
\]
The differential maps this basis to $x^4,x^3y,x^2y^2,xy^3,y^4$, so
\begin{equation}\label{eq:cell-P7}
 D:\mathcal P_7\xrightarrow{\cong}R_4.
\end{equation}
For $w\in\Ahat^{12}$, define $\eta(w)\in\Q$ by
\[
 \widehat\varphi(w)=\eta(w)a^2,
 \qquad F=(\eta,\ell):\Ahat^{12}\longrightarrow\Q^2.
\]
Here $\ell=(u_0^2)^*+(au_0)^*$ is the coefficient functional of \eqref{eq:ell}. Corollary~\ref{cor:cohomology} and Proposition~\ref{prop:completion} show that $F$ annihilates boundaries and induces an isomorphism on $H^{12}(\Ahat)$. More explicitly,
\begin{equation}\label{eq:cell-detectors}
 F(a^2)=(1,0),\qquad F(\Omega)=(0,-1/2),\qquad
 Z^{12}(\Ahat)\cap\ker F=D\Ahat^{11}.
\end{equation}

\begin{lemma}\label{lem:finite-cell-quotient}
The graded subspace $\mathcal J\subset\Ahat$ defined by
\begin{equation}\label{eq:cell-ideal}
 \mathcal J^q=
 \begin{cases}
 0,&q\le6,\\
 \mathcal P_7,&q=7,\\
 R_4\oplus\Span\{xu_i,yu_i:0\le i\le2\},&q=8,\\
 \Ahat^q,&9\le q\le11,\\
 \ker F,&q=12,\\
 \Ahat^q,&q\ge13
 \end{cases}
\end{equation}
is an acyclic differential ideal. Consequently $q_C:\Ahat\twoheadrightarrow C=\Ahat/\mathcal J$ is a quasi-isomorphism.
\end{lemma}
\begin{proof}
Equation~\eqref{eq:cell-P7} and the fact that $F$ annihilates boundaries give $D\mathcal J\subset\mathcal J$. For multiplication, the components in degrees nine through eleven and above twelve are full. Since $\Ahat^1=\Ahat^3=0$, the only additional checks concern products of total degree twelve:
\[
 \mathcal J^7\Ahat^5,\qquad
 \mathcal J^8\Ahat^4,\qquad
 \mathcal J^{10}\Ahat^2.
\]
Each is contained in $x\Ahat^{10}+y\Ahat^{10}$. Indeed, $\Ahat^5=\Span(t_0,t_1,t_2,t_3)$, $\Ahat^4=R_2$, and $\Ahat^2=R_1$. Both $\eta$ and $\ell$ annihilate this subspace by \eqref{eq:completed-properties}. This proves the ideal property.

We verify acyclicity degree by degree. In degree seven, $D|_{\mathcal P_7}$ is injective. The map
\[
 (f_0,f_1,f_2)\longmapsto
 \sum_{i=0}^2 f_i(yt_i-xt_{i+1})
\]
is injective over $R$: the coefficients of $t_0,t_1,t_2$ successively force $f_0,f_1,f_2$ to vanish. Thus $D$ is injective on $R_1\Span(u_0,u_1,u_2)$, and $Z^8(\mathcal J)=R_4=D\mathcal J^7$. Also
\[
 \Ahat^8=\mathcal J^8\oplus\Span(ax,ay),\qquad
 D(ax)=D(ay)=0.
\]
Since $H^9(\Ahat)=0$, this gives $Z^9(\mathcal J)=D\mathcal J^8$. In degrees ten and eleven, acyclicity follows from $H^{10}(\Ahat)=H^{11}(\Ahat)=0$. In degree twelve, \eqref{eq:cell-detectors} gives $Z^{12}(\mathcal J)=D\mathcal J^{11}$.

If $z\in Z^{13}(\mathcal J)$, then $z=Dw$ for some $w\in\Ahat^{12}$ because $H^{13}(\Ahat)=0$. Subtracting a suitable linear combination of the cocycles $a^2$ and $\Omega$ makes $F(w)=0$ without changing $Dw$. Hence $z\in D\mathcal J^{12}$. In higher degrees, the complex agrees with $\Ahat$, whose cohomology vanishes. This proves the lemma.
\end{proof}

The quotient has dimension twenty-five. Using the same letters for the images of the low-degree generators, a homogeneous basis is
\begin{equation}\label{eq:finite-C-basis}
\begin{array}{c|l}
q&\text{basis of }C^q\\ \hline
0&1\\
2&x,y\\
4&x^2,xy,y^2\\
5&t_0,t_1,t_2,t_3\\
6&a,u_0,u_1,u_2,x^3,x^2y,xy^2,y^3\\
7&yt_0,yt_1,yt_2\\
8&ax,ay\\
12&\zeta=a^2,\quad\omega=u_0^2=au_0.
\end{array}
\end{equation}
All omitted components are zero. Its nonzero differentials on the displayed basis are
\begin{equation}\label{eq:finite-C-differential}
 Dt_i=x^{3-i}y^i\quad(0\le i\le3),\qquad
 Du_i=yt_i\quad(0\le i\le2).
\end{equation}
Besides the polynomial products through $R_3$, the nonzero products of positive-degree basis elements are $yt_i$ for $i\le2$, $ax,ay$, and
\begin{equation}\label{eq:finite-C-top-products}
 a^2=\zeta,\qquad au_0=u_0^2=\omega.
\end{equation}
These statements specify multiplication by graded commutativity and bilinearity. For example, $xt_i=yt_3=xu_i=yu_i=0$. Formula~\eqref{eq:finite-C-top-products} follows directly from $F(a^2)=(1,0)$ and $F(au_0)=F(u_0^2)=(0,1)$; all other products of total degree twelve have both coefficients zero. Notice that $\omega$ is a cocycle in $C$, although $u_0^2$ need not be a cocycle in $\Ahat$. In fact $q_C(\Omega)=-\omega/2$.

\subsection{The minimal Quillen model}
Write $\mathbb L(V)$ for the free graded Lie algebra on $V$. We use the Lie cobar construction on the dual coalgebra of a finite-dimensional commutative algebra; see \cite[Sections~1.2.2 and~2.1]{BFMT2020}. For a homogeneous element $z$ of the reduced dual coalgebra, put $v_z=s^{-1}z$. We use the homological differential $\partial$ defined by $(\partial z)(w)=z(Dw)$ and the Lie cobar convention of \cite[Section~2.1, p.~54]{BFMT2020}. With reduced coproduct $\overline\Delta z=\sum z'\otimes z''$, this gives
\begin{equation}\label{eq:Lie-cobar-sign}
 d v_z=-v_{\partial z}
       +\tfrac12\sum(-1)^{|z'|}[v_{z'},v_{z''}].
\end{equation}
Thus the dual of a square of an even cochain contributes one half of a Lie square.

\begin{proposition}\label{prop:finite-Lie-model}
A minimal Quillen model of $Y$ is
\[
 L_Y=\bigl(\mathbb L(\alpha,\beta,e_0,e_1,e_2,c,f_0,f_1,g,h),d\bigr),
\]
with degrees
\[
 |\alpha|=|\beta|=1,\quad |e_i|=3,\quad |c|=5,
 \quad |f_i|=7,\quad |g|=|h|=11.
\]
Set $\lambda=[\beta,[\alpha,e_0]]$. The differential is
\begin{equation}\label{eq:ten-generator-Lie}
\begin{aligned}
 d\alpha=d\beta=dc&=0,\\
 de_0&=\tfrac12[\alpha,\alpha],&
 de_1&=[\alpha,\beta],&
 de_2&=\tfrac12[\beta,\beta],\\
 df_0&=[\alpha,c],&df_1&=[\beta,c],\\
 dg&=\tfrac12[c,c],&
 dh&=[c,\lambda]+\tfrac12[\lambda,\lambda].
\end{aligned}
\end{equation}
\end{proposition}
\begin{proof}
Apply \eqref{eq:Lie-cobar-sign} to $C^\vee$. Its reduced part has dimension twenty-four, so the Lie cobar has twenty-four generators: ten will remain, and fourteen will form seven contractible pairs. The generators corresponding to $x,y,x^2,xy,y^2,a,ax,ay,\zeta,\omega$ are denoted by $\alpha,\beta,e_0,e_1,e_2,c,f_0,f_1,g,h_0$, respectively. The remaining generators are
\[
 T_i\ (|T_i|=4),\quad \mathsf r_i\ (|\mathsf r_i|=5)\quad(0\le i\le3),
 \qquad U_i\ (|U_i|=5),\quad S_i\ (|S_i|=6)\quad(0\le i\le2),
\]
corresponding to $t_i,x^{3-i}y^i,u_i,yt_i$. Define the quadratic Lie words
\begin{equation}\label{eq:Lie-Bi}
\begin{aligned}
 B_0&=[\alpha,e_0],&
 B_1&=[\beta,e_0]+[\alpha,e_1],\\
 B_2&=[\beta,e_1]+[\alpha,e_2],&
 B_3&=[\beta,e_2].
\end{aligned}
\end{equation}
The coproduct dual to the multiplication in \eqref{eq:finite-C-basis} gives
\begin{equation}\label{eq:Lie-before-cancellation}
\begin{aligned}
 d\mathsf r_i&=-T_i+B_i,&dT_i&=0,\\
 dS_i&=-U_i+[\beta,T_i],&dU_i&=0,\\
 dh_0&=[c,U_0]+\tfrac12[U_0,U_0].
\end{aligned}
\end{equation}
The differentials of $\alpha,\beta,e_i,c,f_i,g$ are already those of \eqref{eq:ten-generator-Lie}. The linear terms in \eqref{eq:Lie-before-cancellation} come from \eqref{eq:finite-C-differential}; the two terms in $dh_0$ come from $au_0=u_0^2=\omega$. These lists exhaust the nonzero products and differentials of $C$.

The graded Jacobi identity gives $dB_i=0$. Make the triangular changes
\[
 T'_i=T_i-B_i,\qquad
 \lambda_i=[\beta,B_i],\qquad U'_i=U_i-\lambda_i,
 \qquad S'_i=S_i-[\beta,\mathsf r_i],
\]
with the indicated ranges of indices. They give
\[
 d\mathsf r_i=-T'_i,\quad dT'_i=0,\qquad
 dS'_i=-U'_i,\quad dU'_i=0.
\]
In particular, $\lambda_0=\lambda$ and $d\lambda=0$. Finally put
\begin{equation}\label{eq:Lie-change-top}
 h=h_0-[c,S'_0]-[\lambda,S'_0]+\tfrac12[S'_0,U'_0].
\end{equation}
Using $|c|=|\lambda|=|U'_0|=5$ and $|S'_0|=6$, the derivation rule gives
\[
 dh=[c,\lambda]+\tfrac12[\lambda,\lambda].
\]
For instance, $d[c,S'_0]=[c,U'_0]$ and $d[S'_0,U'_0]=-[U'_0,U'_0]$, which cancel exactly the terms containing $U'_0$.

The resulting free differential Lie algebra is the free product of the displayed $L_Y$ with the free Lie algebra on seven contractible pairs. Here is an explicit contraction of the kernel of projection onto $L_Y$. Define a degree-$+1$ Lie derivation $H$ by
\[
 H(T'_i)=-\mathsf r_i,\qquad H(U'_i)=-S'_i,\qquad
 H(\mathsf r_i)=H(S'_i)=0,\qquad H|_{L_Y}=0.
\]
The degree-zero derivation $dH+Hd$ is the identity on every generator belonging to a contractible pair and is zero on the generators of $L_Y$. It therefore multiplies a Lie word containing $k$ letters from the pairs by $k$. Both $d$ and $H$ preserve this number. On the subspace with $k>0$, the map $H/k$ is a contraction. Thus the kernel is acyclic and the projection is a quasi-isomorphism. It follows that $L_Y$ models $C$, and hence $\Ahat$, by Lemma~\ref{lem:finite-cell-quotient}. Every differential in \eqref{eq:ten-generator-Lie} is decomposable, so this model is minimal.
\end{proof}

\subsection{The inclusion and projection in Quillen models}
The source admits the finite formal model
\[
 B_{\mathrm f}=\bigl(\Lambda(\bar a)/(\bar a^3),0\bigr),\qquad |\bar a|=6,
\]
with quasi-isomorphism $\vartheta:B\to B_{\mathrm f}$ given by $a\mapsto\bar a$ and $b\mapsto0$. Its Quillen model is
\[
 L_X=\bigl(\mathbb L(c,g),d\bigr),\qquad
 |c|=5,\quad |g|=11,\quad dc=0,\quad dg=\tfrac12[c,c].
\]

\begin{proposition}\label{prop:Quillen-splitting}
The maps $p:X\to Y$ and $r:Y\to X$ of Theorem~\ref{thm:main} are represented by
\begin{equation}\label{eq:Quillen-split-diagram}
 L_X\xrightarrow{\ \iota\ }L_Y\xrightarrow{\ \pi\ }L_X,
 \qquad \pi\iota=\id_{L_X},
\end{equation}
where $\iota$ is the inclusion on $c,g$ and
\begin{equation}\label{eq:Quillen-projection}
 \pi(c)=c,\quad\pi(g)=g,\qquad
 \pi(\alpha)=\pi(\beta)=\pi(e_i)=\pi(f_i)=\pi(h)=0.
\end{equation}
In particular the splitting is strict in differential Lie algebras.
\end{proposition}
\begin{proof}
The composite $\vartheta\widehat\varphi$ annihilates $\mathcal J$, including its degree-twelve part by the condition $\eta=0$. It therefore descends to $\bar\varphi:C\to B_{\mathrm f}$. The assignment $\bar\sigma(\bar a)=q_C(a)$ defines a morphism $\bar\sigma:B_{\mathrm f}\to C$, since $C^{18}=0$. With $\sigma:B\hookrightarrow\Ahat$, the identities
\[
 \bar\varphi q_C=\vartheta\widehat\varphi,
 \qquad q_C\sigma=\bar\sigma\vartheta,
 \qquad\bar\varphi\bar\sigma=\id_{B_{\mathrm f}}
\]
identify these morphisms with the original split Sullivan representative.

Dualization and the Lie cobar construction turn $\bar\varphi$ into the inclusion on the generators dual to $\bar a,\bar a^2$, namely $c,g$. They turn $\bar\sigma$ into the projection retaining only $c,g$. The changes of generators in the proof of Proposition~\ref{prop:finite-Lie-model} preserve these maps: every correction contains a generator annihilated by the projection. This proves that \eqref{eq:Quillen-split-diagram} represents $p,r$.

The differential identities can also be checked directly. The inclusion respects $dg=\tfrac12[c,c]$, while $\pi(\lambda)=0$ gives $\pi(dh)=0=d\pi(h)$. All other equations follow immediately from \eqref{eq:ten-generator-Lie}. The composite fixes both generators of $L_X$.
\end{proof}

\subsection{Cell attachments and a finite CW representative}
We use the correspondence between a cell of dimension $n$ and a Lie generator of degree $n-1$, whose differential represents its attaching map; see \cite[p.~117]{BFMT2020}. The following construction is first made in the rational homotopy category. An ordinary finite CW representative is obtained at the end of the proof.

\begin{theorem}\label{thm:eleven-cells}
There is a simply connected finite CW complex $Y_{\mathrm{fin}}$ of dimension twelve with eleven cells, distributed as follows:
\begin{equation}\label{eq:cell-count}
\begin{array}{c|rrrrrr}
 \text{dimension}&0&2&4&6&8&12\\ \hline
 \text{number of cells}&1&2&3&1&2&2.
\end{array}
\end{equation}
Its rationalization is homotopy equivalent to $Y$. More precisely, set
\begin{equation}\label{eq:cell-skeleta}
\begin{aligned}
 Y_2&=S^2\vee S^2,\\
 Y_4&=(\mathbb{CP}^2\times\mathbb{CP}^2)^{(4)},\\
 Y_8&=Y_4\cup_{Y_2}(Y_2\times S^6)
      =(Y_4\times S^6)^{(8)}.
\end{aligned}
\end{equation}
Let $c\in\pi_6(Y_8)\otimes\Q$ be the class of the $S^6$ factor. Let $\xi\in\pi_5(Y_4)\otimes\Q$ come from the first $\mathbb{CP}^2$ factor and be normalized to correspond to the Lie cycle $[\alpha,e_0]$. Put $\lambda=[\beta,\xi]\in\pi_6(Y_4)\otimes\Q$, and also denote its image in $\pi_6(Y_8)\otimes\Q$ by $\lambda$. Then $Y$ is obtained rationally from $Y_8$ by two twelve-dimensional cells with attaching classes
\begin{equation}\label{eq:top-cell-attachments}
 \theta_0=\tfrac12[c,c],\qquad
 \theta_1=\tfrac12[c+\lambda,c+\lambda]
 \quad\text{in }\pi_{11}(Y_8)\otimes\Q.
\end{equation}
These are rational attaching classes; ordinary integral multiples are chosen at the end of the proof to construct $Y_{\mathrm{fin}}$.
\end{theorem}
\begin{proof}
The standard CW structure on $Y_4$ has two two-dimensional cells and three four-dimensional cells. The latter attach by the two Hopf maps and the Whitehead product between the two sphere inclusions. Its rational Lie model is generated by $\alpha,\beta,e_0,e_1,e_2$, with differentials as in \eqref{eq:ten-generator-Lie}. The pushout in \eqref{eq:cell-skeleta} adds a trivially attached six-dimensional cell and two eight-dimensional cells with attaching maps $[\alpha,c]$ and $[\beta,c]$. Hence its Lie model is precisely the subalgebra of $L_Y$ generated by $\alpha,\beta,e_i,c,f_i$.

The cycle $[\alpha,e_0]$ has Lie degree four; it is closed because $[\alpha,[\alpha,\alpha]]=0$. To test whether it is a boundary in $\mathbb L(\alpha,e_0)$, we therefore examine Lie degree five. That component is spanned by $[\alpha,[\alpha,e_0]]$, whose differential is zero. Moreover, $[\alpha,e_0]$ is nonzero: its image in the tensor algebra is $\alpha e_0+e_0\alpha$. Thus it is not a boundary and represents a nonzero class in $\pi_5(\mathbb{CP}^2)\otimes\Q$. It can be identified with a rational normalization of the Hopf map $S^5\to\mathbb{CP}^2$. The projection of $Y_4$ onto its first $\mathbb{CP}^2$ factor splits the inclusion of that factor, so the class remains nonzero in $Y_4$. Consequently the geometric class $[\beta,\xi]$ is represented by $[\beta,[\alpha,e_0]]$. The generator $e_0$ itself is not a cycle; the use of the cycle $[\alpha,e_0]$ is essential in this interpretation.

Since $c$ and $\lambda$ both have Lie degree five, their bracket is symmetric. Replacing the top generator $h$ by $g+h$ in \eqref{eq:ten-generator-Lie} gives
\[
 dg=\tfrac12[c,c],\qquad
 d(g+h)=\tfrac12[c+\lambda,c+\lambda].
\]
Thus the two rational cell attachments in \eqref{eq:top-cell-attachments} produce the model $L_Y$, proving the claimed rational equivalence.

To obtain an ordinary finite CW complex, choose positive integers $N_0,N_1$ and maps $\theta'_i:S^{11}\to Y_8$ whose rational homotopy classes are $N_i\theta_i$. Such choices exist because $\pi_{11}(Y_8)\otimes\Q$ is the rationalization of $\pi_{11}(Y_8)$. Attach two ordinary twelve-dimensional cells along $\theta'_0,\theta'_1$. Their rational Lie differentials are $N_0\theta_0,N_1\theta_1$; rescaling each new generator by $N_i^{-1}$ gives the preceding rational model. The resulting $Y_{\mathrm{fin}}$ therefore has rationalization $Y$ and exactly the cells in \eqref{eq:cell-count}. No cell occurs in an odd dimension, so the cellular boundary is zero, consistently with Corollary~\ref{cor:cohomology}.
\end{proof}

In the cellular basis $g,g+h$, the projection of Proposition~\ref{prop:Quillen-splitting} satisfies
\[
 \pi(g)=g,\qquad \pi(g+h)=g.
\]
Thus both upper attaching classes project to $\tfrac12[c,c]$. In the original basis $g,h$, the same projection fixes $g$ and annihilates $h$. The submodel $L_X=\mathbb L(c,g)$ corresponds to the rational subcomplex consisting of the six-dimensional sphere and the first twelve-dimensional cell. The inclusion, projection and splitting are identified here in the rational homotopy category.

\begin{corollary}\label{cor:target-hyperbolic}
The rationalization of $Y_{\mathrm{fin}}$ is hyperbolic. In particular, every minimal Sullivan model of $Y$ has infinitely many generators, although $Y$ admits the ten-generator minimal Quillen model of Proposition~\ref{prop:finite-Lie-model}.
\end{corollary}
\begin{proof}
The cohomology of $Y$ is finite dimensional, vanishes above degree twelve, and has dimension two in degree twelve. It cannot satisfy Poincar\'e duality: a connected Poincar\'e duality algebra has a one-dimensional top component. A rationally elliptic simply connected space has Poincar\'e duality cohomology \cite[Chapter~32]{FHT2001}. Thus $Y$ is not elliptic. Since its total rational cohomology is finite dimensional, this means that $\pi_*(Y)\otimes\Q$ is infinite dimensional, so $Y$ is hyperbolic. For a minimal Sullivan model $(\Lambda V,D)$, $V^n\cong\operatorname{Hom}(\pi_n(Y),\Q)$; this proves the final assertion. The generators of a minimal Quillen model instead correspond to $s^{-1}\widetilde H_*(Y;\Q)$ \cite[Section~1.2.2, p.~34]{BFMT2020}, so there is no conflict with their finite number.
\end{proof}


\clearpage
\begin{thebibliography}{99}
\raggedright
\bibitem{BFMT2020}
U.~Buijs, Y.~Félix, A.~Murillo and D.~Tanré, \emph{Lie Models in Topology},
Progress in Mathematics, vol.~335, Birkhäuser/Springer, Cham, 2020.
\href{https://doi.org/10.1007/978-3-030-54430-0}{doi:10.1007/978-3-030-54430-0}.

\bibitem{CV2017}
J.~G. Carrasquel-Vera, \emph{The rational sectional category of certain maps},
Ann. Sc. Norm. Super. Pisa Cl. Sci. (5) \textbf{17} (2017), no.~2, 805--813.
\href{https://doi.org/10.2422/2036-2145.201510_020}{doi:10.2422/2036-2145.201510\_020}.

\bibitem{FGKV2006}
L.~Fernández Suárez, P.~Ghienne, T.~Kahl and L.~Vandembroucq,
\emph{Joins of DGA modules and sectional category},
Algebr. Geom. Topol. \textbf{6} (2006), 119--144.
\href{https://doi.org/10.2140/agt.2006.6.119}{doi:10.2140/agt.2006.6.119}.

\bibitem{FHT2001}
Y.~Félix, S.~Halperin and J.-C. Thomas, \emph{Rational Homotopy Theory},
Graduate Texts in Mathematics, vol.~205, Springer-Verlag, New York, 2001.
\href{https://doi.org/10.1007/978-1-4613-0105-9}{doi:10.1007/978-1-4613-0105-9}.

\bibitem{Hess1991}
K.~P. Hess, \emph{A proof of Ganea's conjecture for rational spaces},
Topology \textbf{30} (1991), no.~2, 205--214.
\href{https://doi.org/10.1016/0040-9383(91)90006-P}{doi:10.1016/0040-9383(91)90006-P}.

\bibitem{Schwarz1966}
A.~S. Schwarz, \emph{The genus of a fiber space},
Amer. Math. Soc. Transl. (2) \textbf{55} (1966), 49--140.

\bibitem{Stanley2000}
D.~Stanley, \emph{The sectional category of spherical fibrations},
Proc. Amer. Math. Soc. \textbf{128} (2000), no.~10, 3137--3143.
\href{https://doi.org/10.1090/S0002-9939-00-05468-X}{doi:10.1090/S0002-9939-00-05468-X}.

\bibitem{Sullivan1977}
D.~Sullivan, \emph{Infinitesimal computations in topology},
Publ. Math. Inst. Hautes Études Sci. \textbf{47} (1977), 269--331.
\href{https://doi.org/10.1007/BF02684341}{doi:10.1007/BF02684341}.
\end{thebibliography}
\end{document}